\documentclass[preprint,12pt]{elsarticle}

\usepackage{amssymb}
\usepackage[colorlinks=true]{hyperref}
\usepackage{geometry}
\usepackage{amsmath}
\usepackage{booktabs}
\usepackage[all]{xy}
\usepackage{amsmath}
\usepackage{accents}
\newlength{\dhatheight}

\usepackage{latexsym}
\usepackage{mathrsfs}
\usepackage{amsthm}
\usepackage{graphicx}
\usepackage{amssymb}
\usepackage{amsfonts}
\usepackage{amsbsy}

\usepackage[shortlabels]{enumitem}
\usepackage{url}
\usepackage{array}
\usepackage{pdflscape}
\usepackage{xcolor}
\usepackage{stmaryrd}

\usepackage{verbatim}
\usepackage{epsfig}
\usepackage{bookmark}
\newtheorem{thm}{Theorem}[section]
\newtheorem{lem}[thm]{Lemma}
\newtheorem{cor}[thm]{Corollary}

\newtheorem{prop}[thm]{Proposition}

\newtheorem{defn}[thm]{Definition}
\newtheorem{remark}[thm]{Remark}

\def\tr{\operatorname{tr}}

\begin{document}

\begin{frontmatter}



\title{Quantum chromatic numbers of some graphs in Hamming schemes}


\author{Xiwang Cao$^{a,}$\footnote{The research of X. Cao is supported by National Natural Science Foundation of China, Grant No. 12171241. The research of K. Feng is supported by National Natural Science Foundation of China, Grant No. 12031011. The research of Y. Tan is supported by National Natural Science Foundation of China, Grant No. 12371339}, Keqin Feng$^{b}$, Ying-Ying Tan$^c$}

\address{$^{a}$School of Mathematics, Nanjing University of Aeronautics and Astronautics, China\\ $^{b}$Department of Mathematics, Tsinghua University, China\\
$^c$Anhui Jianzhu University, China}
\begin{abstract}
The study of quantum chromatic numbers of graphs is a hot research topic in recent years. However, the infinite family of graphs with known quantum chromatic numbers are rare, as far as we know, the only known such graphs (except for complete graphs, cycles, bipartite graphs and some trivial cases) are the Hadamard graphs $H_n$ with $2^n$ vertices and $n$ a multiple of $4$. In this paper, we consider the graphs in Hamming schemes, we determined the quantum chromatic numbers of one class of such graphs. Notably, this is the second known family of graphs whose quantum chromatic numbers are explicitly determined except for some cases aforementioned. We also provide some bounds for the quantum chromatic numbers of some other graphs in Hamming schemes. Consequently, we can obtain the quantum chromatic numbers of products of some graphs.

\end{abstract}


\end{frontmatter}


\section{Introduction}
\label{intro}
In recent years, combinatorial designs and graph theory have become useful tools in the study of quantum communications and quantum information processing, mainly reflected in the following aspects:

\begin{itemize}
  \item Quantum states constructed from graphs and hypergraphs are used to study entanglement phenomena and construct high-performance quantum error correction codes \cite{AHKS06,cameron};
  \item Spherical designs are used to construct various types of unbiased bases for quantum measurement \cite{feng};
  \item Classical combinatorial designs are extended to quantum versions (quantum orthogonal Latin squares, etc.) to study and construct various maximally entangled quantum states\cite{CW};
  \item Quantum state transfer are employed for transmitting quantum information using quantum networks, for example, the so-called perfect state transfer, uniform mixing. etc.\cite{Ada};
   \item Extend some parameters and concepts of the classical graphs to that of the quantum version, such as quantum homomorphism, quantum chromatic numbers, quantum independent numbers, etc.\cite{{cameron}}
\end{itemize}

Let $\Gamma$ be a simple graph whose vertex set is $V$ and edge set $E$. A colouring on $\Gamma$ is an assignment of colors to vertices of the graph such that the two vertices of each edge have different colours. Graph colouring is of particular interesting since it finds applications in quantum information theory and communication as seen in \cite{AHKS06}. Classical graph colouring can be interpreted as a so-called non-local game, where two players Alice and Bob collaborate to answer pairs of questions without communication using some prior agreed-upon strategy. Quantum coloring of graphs is a modification of the classical graph coloring where the players may use ``quantum" strategies meaning a
shared entangled resource is available. In the colouring games, we are interested in the minimum number of the colours needed to win. For a classical graph $\Gamma$, this minimum number is denoted by $\chi(\Gamma)$, and termed as the chromatic number. For quantum graphs, by $\chi_q(\Gamma)$ and called the quantum chromatic number. The mathematical definition of these notations will be given in the next section. Notably, it was
shown the tremendous quantum advantage in \cite{AHKS06} that Hadamard graphs can provide examples of an exponential gap between the quantum and classical chromatic numbers using a result in \cite{PF}. However, in general, it is difficult to determine the chromatic number of a given graph, even more harder to evaluate or estimate the quantum chromatic number of the graph. In \cite{ji}, Ji proved that determining these numbers is NP hard. Up to date, except for complete graphs, cycles, bipartite graphs, Hadamard graphs $H_n$ and some trivial cases, only a few lower bounds on the chromatic numbers are known for some sporadic graphs, as far as we know, the only non-trivial family of graphs whose quantum chromatic numbers are explicitly determined is the class of Hadamard graphs which were defined by Ito \cite{Ito} in the year 1985. Using the representations of certain groups and extensive computations, Ito obtained the spectra of the Hadamaed graphs. Very recently, Menamara \cite{Mena} also calculated the quantum chromatic numbers of the Hadamard graphs of order
$N = 2^n$ for $n$ a multiple of $4$ using character sums over finite fields and the upper bound derived by Avis etal \cite{AHKS06}, as well
as an application of the Hoffman-like lower bound of Elphick and Wocjan \cite{CW} that was generalized by Ganesan \cite{Gan} for quantum graphs. One of the main results in \cite{Mena} is as follows:

\begin{thm}\cite{Mena}\label{thm-1} (Exact quantum chromatic number of Hadamard graphs). Let $H_n$ be the Hadamard graph
on $2^n$ vertices, $n$ a multiple of 4. Then,
\begin{equation}\label{f-1}
\chi_q(H_n ) = n.
\end{equation}
\end{thm}
We note that the above result is already known, see for example \cite{CW}.  Menamara \cite{Mena} gave a new proof of it by providing an explicit quantum colouring of $H_n$.

In this paper, we give a new method for calculating the spectrum of the Hadamard graph in Theorem \ref{thm-1} by using some properties of Krawchouk polynomials, we also determined the quantum chromatic numbers of some other graphs in Hamming schemes, some bounds on the quantum chromatic numbers of some graphs in Hamming schemes are provided.

The organization of the paper is as follows: In Sect. \ref{prelim}, we give some backgrounds on quantum information and quantum measurements, as well as some basic concepts of graph theory. In Sect. \ref{main results}, we consider the graphs in Hamming schemes.  Using some spectral bounds on the quantum chromatic numbers, we successfully obtained the quantum chromatic numbers of one class of such graphs, see Theorem \ref{main-1}. Some bounds of quantum chromatic numbers of other class of graphs in Hamming shcems are provided as well (Theorem \ref{thm-3.5} and Proposition \ref{prop-3.6}). By utilizing the products of graphs, we can also get the quantum chromatic numbers of some graphs (Theorem \ref{thm-3.11}).



\section{Preliminaries}\label{prelim}

\subsection{Some basic concepts of quantum communication}\label{Some basic concepts of quantum communication}
\subsubsection{ Quantum state}

In digital communications, information is represented as an $K$-tuple $c=(c_0,c_1,\cdots,c_{K-1})$, where the entries $c_i\in Q$ which is a $q$-ary set. In most cases, $Q$ is chosen as the finite field $\mathbb{F}_q$ with $q$ elements, or a cyclic group $\{0,1,\cdots,q-1\} \pmod q$. Then $c$ can be viewed as a vector in $Q^K$.

In quantum communication, each qubit, denoted by $|v\rangle =(v_0,v_1,\cdots, v_{K-1})^T$, is a unite vector in the $K$-dimensional vector space $\mathbb{C}^K$. For every $|v\rangle =(v_0,v_1,\cdots, v_{K-1})^T$, $|u\rangle =(u_0,u_1,\cdots, u_{K-1})^T\in \mathbb{C}^K$, define the inner product
\begin{equation*}
  \langle u|v\rangle=\sum_{i=0}^{K-1}u_iv_i^*.
\end{equation*}
If $\langle u|v\rangle=0$, we say that $|u\rangle$ and $|v\rangle$ are separable. A quantum state is a vector in the space $\mathbb{C}^{d_1}\otimes \mathbb{C}^{d_2}\otimes \cdots \otimes\mathbb{C}^{d_K}$ which is the tensor product of complex spaces. Take $V_i=\mathbb{C}^{d_i}$, $1\leq i\leq K$, and choose an orthnormal basis of $V_i$ as $|0\rangle, |1\rangle, \cdots, |d_{i}-1\rangle$. Then

$$\{|e_1\rangle\otimes\cdots\otimes|e_K\rangle: 0\leq e_i\leq d_i-1, (1\leq i\leq K)\}$$
forms an orthnormal basis of $\mathfrak{V}:=\mathbb{C}^{d_1}\otimes \mathbb{C}^{d_2}\otimes \cdots \otimes\mathbb{C}^{d_K}$. Thus each quantum state in $\mathfrak{V}$ can be uniquely represented as
\begin{equation*}
  |v\rangle=\sum_{0\leq e_i\leq d_i-1,1\leq i\leq K}a_{e_1,\cdots,e_K}|e_1\rangle\otimes\cdots\otimes|e_K\rangle, a_{e_1,\cdots,e_K}\in \mathbb{C}.
\end{equation*}

\subsubsection{ Quantum measurement}

Let $H=(h_{ij})_{0\leq i,j\leq K-1}$ be a Hermite matrix. Then the quantum measurement of $H$ on $|v\rangle\in \mathbb{C}^K$ is defined by $H|v\rangle$. In quantum communication, $H$ can be written as $H=\sum_{i,j=0}^{K-1}h_{ij}|i\rangle \langle j|$.

Generally speaking, it is not easy to devise a measurement procedure which uniquely identifies the given quantum state from the statistical date produced by the measurements. For example, if the state of the quantum system is given by an $K \times K$
density matrix, the complete measurement statistics of one fixed {\it von
Neumann} measurement is not sufficient to reconstruct the state, see e.g. \cite{kla}. However, it is possible to perform a somewhat general measurement procedure on a quantum system, namely a {positive operator-valued measurement} (or POVM for short), see \cite{peres}. Mathematically, a POVM is a collection of some semi-positive
operators $E_i \geq 0$, each $E_i$ is a $K$ by $K$ matrix, called POVM elements, satisfying the summation of all these operators is equal to $I_K$ the identity matrix. POVMs constitute a basic ingredient in many applications of quantum information processing: quantum tomography, quantum key distribution required in cryptography, discrete Wigner function, quantum teleportation, quantum error correction codes, dense coding, teleportation, entanglement swapping, covariant cloning and so on, see for example \cite{NC}.

\subsubsection{Projective measurement}

In a quantum measurement, people usually use some projective matrices $P=(p_{ij})_{1\leq i,j\leq K}: \mathbb{C}^K\rightarrow \mathbb{C}^K$. A Hermite matrix $P$ is called projective if $P^2=P=P^*$. Suppose that $|v\rangle$ is contained in the image of $P$, that is, there is a vector $|a\rangle\in \mathbb{C}^K$ such that $P|a\rangle=|v\rangle$. Then
\begin{equation*}
  P|v\rangle=P^2|a\rangle=P|a\rangle=|v\rangle.
\end{equation*}
Thus $P|_{{\rm Im}(P)}={\rm id}$. Then there exists a unitary matrix $U$ such that $U^*PU={\rm diag}(I_r,0)$, where $r={\rm rank}(P)$.

Finally, a set of projective operators $\{P_1,P_2,\cdots, P_K\}$ in $\mathbb{C}^{K\times K}$ is called a complete POVM if $P_iP_j=0_K$ for every $1\leq i\neq j\leq K$, and $\sum_{i=1}^KP_i=I_K$. In this case, it can be proved that there exists a unitary matrix $U$ such that
\begin{equation*}
  U^*P_iU={\rm diag}(0,0,\cdots,1,0,\cdots,0), 1\leq i\leq K,
\end{equation*}
where $1$ is in the $i$-th entry. Moreover, $\mathbb{C}^K={\rm Im}(P_1)\oplus{\rm Im}(P_2)\oplus \cdots \oplus {\rm Im}(P_K).$

\subsection{ Quantum homomorphism of graphs and graph colouring}\label{graph theory}


Let $\Gamma=(V,E)$ be a simple graph with $n=|V|$ vertices and $m=|E|$ edges. A homomorphism $\varphi$ from a graph $\Gamma_1=(V_1,E_1)$ to a graph $\Gamma_2=(V_2,E_2)$ is a mapping $\varphi: \Gamma_1\rightarrow \Gamma_2$ satisfying $(\varphi(u),\varphi(v))\in E_2$ if $(u,v)\in E_1$. For example, if $\Gamma_2=K_c$ is a complete graph on $c$ vertices, then $\varphi: \Gamma=(V, E)\rightarrow K_c$ is a homomorphism means that if $(u,v)\in E$, then $\varphi(u)\neq \varphi(v)$. We name the minimum number $c$ such that there exists a homomorphism from $\Gamma$ to $K_c$ the chromatic number of $\Gamma$ and denote it by $\chi(\Gamma)$. The maximum number $c$ such that there is a homomorphism from $K_c$ to $\Gamma$ is called the clique number of $\Gamma$ and denoted by $\omega(\Gamma)$. Let $\bar{\Gamma}$ be the complement graph of $\Gamma$. Then $\alpha(\Gamma):=\omega(\bar{\Gamma})$ is called the independent number of $\Gamma$.

\begin{defn}A quantum homomorphism from a graph $\Gamma_1=(V_1,E_1)$ to a graph $\Gamma_2=(V_2,E_2)$ means that there is a positive integer $d$ such that for every $x\in V_1$, there exists a complete orthogonal projective system $\mathfrak{F}_x=\{P_{x,y}: y\in V_2\}$ satisfying the following two conditions:

\begin{enumerate}
  \item (Completeness) For every $x\in V_1$, $\mathfrak{F}_x$ is a complete orthogonal system, namely, $P_{x,y}^2=P_{x,y}=P_{x,y}^*$ and, when $y\neq y'$, we have $P_{x,y}P_{x,y'}=0_d$. Moreover, $\sum_{y\in V_2}P_{x,y}=I_d$.
  \item (Orthogonality) For every $x,x'\in V_1$, $y,y'\in V_2$, we have $P_{x,y}P_{x',y'}=0_d$ whence $(x,x')\in E_1, (y,y')\not\in E_2$.
\end{enumerate}
 \end{defn}
It is easy to see that a classical graph homomorphism is actually a quantum homomorphism. We note that, in a recent paper \cite{Ada}, Chan et al gave a definition of quantum isomorphism of graphs and proved that any two Hadamard graphs on the same number of vertices are quantum isomorphic.

\begin{defn}The quantum chromatic number of a graph $\Gamma$, denoted by $\chi_q(\Gamma)$, is the minimum number $c$ such that there exists a quantum homomorphism from $\Gamma$ to the complete graph $K_c$.\end{defn}

By definition, we see that for every graph $\Gamma$,
\begin{equation}\label{f-2}
 \chi_q(\Gamma)\leq \chi(\Gamma).
\end{equation}

It seems that how to determine the quantum chromatic number of a given graph is very hard. Up to date, except for some sporadic graphs with small size of vertices and some trivial cases, the only known class of graphs whose quantum chromatic numbers are determined are the Hadamard graphs $H_n$ with $n$ a multiple of $4$. This situation motivates the study of quantum chromatic numbers of graphs. The following questions are of particular interesting.

\begin{itemize}
  \item For a specific class of graphs, determine their chromatic numbers;
  \item Separable problem: find graphs such that their chromatic numbers are strictly less than their quantum chromatic numbers, note that $\chi_q(H_n)=n$ when $4|n$, however, $\chi(H_n)\geq 2^{n/2}$ when $n$ is bigger enough;
  \item Find some lower or upper bounds for the chromatic numbers of some class of graphs.
\end{itemize}

For more information about quantum chromatic numbers, we refer the reader to \cite{ cameron,CW, feng}.

\subsection{Spectra of Cayley graphs and bounds on quantum chromatic numbers}\label{Spectrum of Cayley graphs}

Let $\Gamma=(V,E)$ be a simple graph with $|V|=n, |E|=m$, $A$ be its adjacency matrix. The spectrum of $A$ is also termed the spectrum of $\Gamma$. For every $x\in V$, the number of its neighborhoods is defined as its valency (or degree). If we label the vertices of $\Gamma$ as $x_1,\cdots,x_n$, and denote the valency of $x_i$ by $k_i$. Then $D:={\rm diag}(k_1,\cdots,k_n)$ is the degree matrix of $\Gamma$. $L=D-A$ (resp. $L^+=D+A$) is the Laplace (resp. signless Laplace) matrix of $\Gamma$. Suppose that the eigenvalues of $A$, $L$ and $L^+$ are $\lambda_1\geq \lambda_2\geq \cdots\geq \lambda_n$, $\theta_1\geq \theta_2\geq \cdots \geq \theta_n(= 0)$, and $\delta_1\geq \delta_2\geq \cdots \delta_n$, respectively. The following result is known, see for example, \cite{CW} and the references therein.

\begin{thm}\label{thm-2.3} Let notations be defined as above. Then
\begin{equation}\label{f-3}
  \chi(\Gamma)\geq 1+\max\left\{\frac{\lambda_1}{|\lambda_n|}, \frac{2m}{2m-n\delta_n},\frac{\lambda_1}{\lambda_1-\delta_1+\theta_1},\frac{n^+}{n^-},\frac{n^-}{n^+},\frac{S^+}{S^-},\frac{S^-}{S^+}\right\},
\end{equation}
where $n^+$ (resp. $n^-$) is the number of positive (resp. negative) eigenvalues of $\Gamma$, and $S^+$ (resp. $S^-$) is the summation of the squares of positive (resp. negative) eigenvalues of $\Gamma$.
\end{thm}
Some quantum versions of Theorem \ref{thm-2.3} are known, for example, a spectral lower bound on the quantum chromatic numbers is provided in \cite{CW}.
\begin{lem}\cite{CW}\label{lem-2.4} For any graph $\Gamma$ with eigenvalues $\lambda_1\geq \lambda_2\geq \cdots \geq \lambda_n$, we have
\begin{equation}\label{f-4'}
  \chi_q(\Gamma)\geq 1+\frac{\lambda_1}{|\lambda_n|}.
\end{equation}
\end{lem}

Let $G$ be a finite group. A { representation} of $G$ is a homomorphism $\rho: G \rightarrow GL(U)$ for a (finite-dimensional) non-zero vector space $U$. The dimension of $U$ is called the { degree} of $\rho$. Two representations $\rho: G\rightarrow GL(U)$ and $\varrho: G\rightarrow GL(W)$ are {\it equivalent}, denoted by $\rho\sim \varrho$, if there exists an isomorphism $T: U\rightarrow W$ such that $\rho_g=T^{-1}\varrho_g T$ for all $g\in G$. For every representation $\rho: G\rightarrow GL(U)$ of $G$, the { character} of $\chi_\rho$ is defined by:
\begin{equation*}
  \chi_\rho: G\rightarrow \mathbb{C}, \chi_\rho(g)=\tr(\rho(g))  \mbox{ for all $g\in G$},
\end{equation*}
where $\tr(\rho(g))$ is the trace of the representation matrix with respect to a basis of $U$.

A subspace $W$ of $U$ is said to be $G$-{invariant} if $\rho(g)\omega\in W$ for every $g\in G$ and $\omega\in W$. Obviously, $\{0\}$ and $U$ are $G$-invariant subspaces, called trivial subspaces. If $U$ has no non-trivial $G$-invariant subspaces, then $\rho$ is called an {irreducible representation} and $\chi_\rho$ an {irreducible character} of $G$.

Let $S$ be a subset of $G$ with $S^{-1}:=\{s^{-1}: s\in S\}=S$. A Cayley graph over $G$ with connection set $S$ is defined by $\Gamma:={\rm Cay}(G,S)$ where the vertex set is $G$ and two elements $x,y\in G$ are adjacent if and only if $xy^{-1}\in S$. If $S$ is conjugation closed meaning that for every $x\in G$ and $s\in S$, we have $x^{-1}sx\in S$. In this case, the Cayley graph ${\rm Cay}(G,S)$ is called normal. For normal Cayley graphs, the following result is well-known.

\begin{lem}\label{lem-2.3}\cite[pp. 69-70]{stein} Let $G=\{g_1,\cdots,g_n\}$ be a finite group of order $n$ and let $\rho^{(1)},\cdots,\rho^{(s)}$ be a complete set of unitary representatives of the equivalent classes of irreducible representations of $G$. Let $\chi_i$ be the character of $\rho^{(i)}$ and $d_i$ be the degree of $\chi_i$. Let $S$ be a symmetric set and further that $gSg^{-1}=S$ for all $g\in G$. Then the eigenvalues of the adjacency matrix $A$ of the Cayley graph ${\rm Cay}(G,S)$ with respect to $S$ are given by
\begin{equation*}
  \lambda_k=\frac{1}{d_k}\sum_{g\in S}\chi_k(g), 1\leq k\leq s,
\end{equation*}
each $\lambda_k$ has multiplicity $d_k^2$. Moreover, the vectors
\begin{equation*}
 v_{ij}^{(k)}=\frac{\sqrt{d_k}}{|G|}\left(\rho_{ij}^{(k)}(g_1),\cdots,\rho_{ij}^{(k)}(g_n)\right)^T, 1\leq i,j\leq d_k
\end{equation*}
form an orthonormal basis for the eigenspace $V_{\lambda_k}$.
\end{lem}
Note that a proof of Lemma \ref{lem-2.3} can also be found in \cite[Theorem 9]{murty}.

As a consequence, if $G$ is a finite abelian group, we assume that $G$ is decomposed as a direct sum of cyclic groups, $G=\mathbb{Z}_{n_1}\oplus \cdots \oplus \mathbb{Z}_{n_r}$, then the spectrum of the Cayley graph $\Gamma={\rm Cay}(G,S)$ is given by
\begin{equation}\label{f-4}
  \lambda_g=\sum_{s\in S}\chi_g(s),
\end{equation}
where $\chi_g(s)=\prod_{j=1}^s\xi_{n_j}^{g_js_j}$, $\forall g=(g_1,\cdots,g_r)\in G$, $s=(s_1,\cdots,s_r)\in S$, and $\xi_{n_j}$ is a primitive $n_j$-th root of unity in $\mathbb{C}$. Of course, (\ref{f-4}) can also be proved by using an elementary method.

\subsection{Krawtchouk polynomials}

For positive integers $n,\ell$, and $q$, the Krawchouk polynomial in variable $x$ is defined by
\begin{equation}\label{f-5}
  K_\ell^{n,q}(x)=\sum_{j=0}^\ell(-1)^j(q-1)^{\ell-j}\tbinom{x}{j}\tbinom{n-x}{\ell-j}.
\end{equation}
Krawchouk polynomials are a kind of orthogonal polynomials and have many important applications in many fields such as coding theory, function analysis and approximation etc. For our purpose, we list some of the properties of such polynomials as follows.
\begin{thm}\cite{Lev}\label{Krawchouk} The Krawchouk polynomials have the following properties.
\begin{enumerate}
  \item (Orthogonality Relations): For every $i,j, (i,j=0,1,\cdots,n)$
  \begin{equation}\label{f-6}
    \sum_{d=0}^nK_i^n(d)K_j^n(d)(q-1)^d\tbinom{n}{d}=q^n(q-1)^i\tbinom{n}{i}\delta_{i,j}.
  \end{equation}
  \item (Recursive Relation): For any $k = 1,\cdots, n$ and any real $x$
  \begin{eqnarray}
    K_k^n(x)&=& K_k^{n-1}(x-1)-K_{k-1}^{n-1}(x-1) \\
    K_k^n(x) &=&  K_k^{n-1}(x)+(q-1)K_{k-1}^{n-1}(x)\\
    K_k^{n-1}(x)&=&\sum_{j=0}^kK_j^n(x)(1-q)^{k-j}.
  \end{eqnarray}
  \item (Reciprocal Law):
  \begin{equation}\label{f-14}
    (q-1)^i\tbinom{n}{i}K_d^n(i)=(q-1)^d\tbinom{n}{d}K_i^n(d).
  \end{equation}
  \item (Generating Function):
  \begin{equation}\label{f-15}
    \sum_{k=0}^{n}K_k^n(d)z^k=(1-z)^{d}(1+(q-1)z)^{n-d}.
  \end{equation}
  \item (Inversion Formula):
  \begin{equation}\label{f-16}
    f(x)=\sum_{j=0}^nf_jK_j^n(x)
  \end{equation}
  if and only if for every $i=0,1,\cdots,n$,
  \begin{equation}\label{f-17}
    f_i=q^{-n}\sum_{j=0}^nf(j)K_j^n(i).
  \end{equation}
\end{enumerate}
\end{thm}
\subsection{Hamming schemes}

Let $q\geq 2, n\geq 1$ be integers, $Q$ be a set of $q$ elements. $Q^n=\{(x_1,x_2,\cdots, x_n): x_i\in Q\}$. For $x=(x_1,x_2,\cdots, x_n), y=(y_1,y_2,\cdots,y_n)\in Q^n$, the Hamming distance of $x,y$, denoted by $d(x,y)$, is the number of coordinates they differ. For every $1\leq \ell\leq n$, the graph $H(n,q,\ell)$ is defined as $H(n,q,\ell)=(V,E)$, where the vertex set $V=Q^n$, two vectors $x,y$ are adjacent if $d(x,y)=\ell$. Let $A_\ell$ be the adjacency matrix of $H(n,q,\ell)$. Then $\{A_\ell: 0\leq \ell\leq n\}$, where $A_0=I_n$, forms an association scheme, named the Hamming scheme. When $q$ is fixed, we write $H(n,q,\ell)$ simply as $H_{n,\ell}$. In this paper, we call $H_{n,\ell}$ a Hamming graph for each $\ell$. The eigenvalues of $A_\ell, 0\leq \ell\leq n$ are well-known. In fact, $H_{n,\ell}$ is a Cayley graph. Let $Q=\{0,1,2,\cdots,q-1\} \pmod q$ be a cyclic group of order $q$, $S=\{x\in Q^n: wt(x)=\ell\}$, where $wt(x)=d(x,0_n)$. Then $H_{n,\ell}={\rm Cay}(Q,S)$. Thus for every $a\in Q^n$, the corresponding eigenvalue is
$\lambda_a=\sum_{x\in S}\xi_q^{a\cdot x}$, where $a\cdot x$ is the inner product of $x$ with $a$, namely, $(x_1,\cdots,x_n)\cdot (a_1,\cdots,a_n)=\sum_{i=1}^nx_ia_i$, and $\xi_q=e^{\frac{2\pi \sqrt{-1}}{q}}$ is a primitive $q$-th root of unity. Write $a=(a_0,\cdots,a_{n-1})$ and $wt(a)=r$. Then
\begin{equation*}
  \lambda_a=\sum_{x=(x_0,\cdots,x_{n-1})\in Q^n, wt(x)=\ell}\xi_q^{\sum_{i=0}^{n-1}x_ia_i}.
\end{equation*}
Since
\begin{equation*}
  \sum_{0\neq x_i\in Q}\xi_q^{x_ia_i}=\left\{\begin{array}{ll}
                                         q-1, & \mbox{ if $a_i=0$}, \\
                                         -1, & \mbox{ if $a_i\neq 0$ },
                                       \end{array}
  \right.
\end{equation*}
we know that
\begin{equation}\label{f-n1}
  \lambda_a=\sum_{j=0}^\ell(-1)^j(q-1)^{\ell-j}\tbinom{r}{j}\tbinom{n-r}{\ell-j}=K_\ell(r).
\end{equation}
Even though we have the above formula for computing the eigenvalues of $H_{n,\ell}$, it is not an explicit expression. In this paper, we will give some concise formulae for eigenvalues of Hamming graphs.

\section{Main results}\label{main results}

Let $V_n=\{(x_0,x_1,\cdots,x_{n-1}): x_i\in \mathbb{F}_2\}$, where $\mathbb{F}_2$ is the binary field. $V_n$ is a $n$-dimensional vector space over $\mathbb{F}_2$. For $x=(x_0,x_1,\cdots,x_{n-1})\in V_n$, the Hamming weight of $x$, denoted by $wt(x)$, is the number of nonzero coordinates of $x$, the support of $x$ is ${\rm supp}(x):=\{i: 0\leq i\leq n-1, x_i=1\}$.
For $x,y\in V_n$, the Hamming distance between $x$ and $y$ is $d(x,y)=wt(x-y)$. The following defined Hadamard graph is isomorphic to that defined by Ito \cite{Ito}.

\begin{defn}Let $n$ be a positive integer with $4|n$. Define the Hadamard graph $H_n=(V_n,E_n)$, where $V_n$ is the $n$-dimensional vector space over $\mathbb{F}_2$, two vectors $x,y\in V_n$ are adjacent if and only if $d(x,y)=n/2$.\end{defn}

In this paper, we consider the graph $H_{n,\ell}$. That is, $H_{n,\ell}=(V_n,E_n^{(\ell)})$, $V_n$ is the $n$-dimensional vector space over $\mathbb{F}_2$, two vectors $x,y\in V_n$ are adjacent if and only if $d(x,y)=\ell$. Obviously, the Hadamard graph $H_n$ is $H_{n,n/2}$.

Note that if $\ell$ is odd, then $H_{n,\ell}$ is a bipartite graph and then its quantum chromatic number is $2$. Thus in next sequel, we assume that $\ell$ is even.

In this section, we first give a simple method to calculate the spectrum of $H_n$ and prove that $\chi_q(H_n)=n$. Then for the Hamming graphs, we present some new results on the quantum chromatic numbers of such graphs.

\subsection{Proof of Theorem \ref{thm-1}}\label{proof of Thm-1}

Firstly, it is easy to see that $H_n={\rm Cay}(V_n,S)$, where $S=\{x\in V_n: wt(x)=n/2\}$. The character group of $V_n$ (as an elementary commutative $2$-group of rank $n$) is $\widehat{V_n}=\{\phi_a: a\in V_n\}$, where $\phi_a(x)=(-1)^{x\cdot a}$, $x\cdot a$ is the inner product of $x$ and $a$, i.e., $x\cdot a=\sum_{i=0}^{n-1}x_ia_i$, $a=(a_0,\cdots,a_{n-1})$.
By (\ref{f-4}), the eigenvalues of $H_n$ are
\begin{equation}\label{f-18}
  \lambda_a=\sum_{s\in S}(-1)^{s\cdot a}, a\in V_n.
\end{equation}
Obviously, $\lambda_{0_n}=|S|=\tbinom{n}{n/2}$. Take $a=1_n:=(1,1,\cdots,1)$. Then
\begin{equation*}
  \lambda_{1_n}=\sum_{s\in S}(-1)^{s\cdot 1_n}=\sum_{s\in S}(-1)^{wt(s)}=\sum_{s\in S}(-1)^{n/2}=\sum_{s\in S}1=|S|=\tbinom{n}{n/2}=\lambda_{0_n}.
\end{equation*}
And for every $a\in V_n$, $a\neq 0_n, 1_n$, then $\lambda_a<\lambda_{0_n}$. Thus $\lambda_{\max}=\tbinom{n}{n/2}$ with multiplicity $2$. $H_n$ has two isomorphic components. Below, we proceed to find the minimum eigenvalue $\lambda_{\min}$.

For $a(\neq 0_n,1_n)\in V_n$,
\begin{equation*}
  \lambda_a=\sum_{s\in S}(-1)^{s\cdot a}=\sum_{x\in V_n: wt(x)=n/2}(-1)^{x\cdot a}.
\end{equation*}
Suppose that $a=(a_0,\cdots,a_{n-1})\in V_n$, $wt(a)=r$, $1\leq wt(a)<n$. Assume that ${\rm supp}(a)=\{i_1,i_2,\cdots,i_r\}$. Let $x$ run through $V_n$ with weight $n/2$. If $|{\rm supp}(x)\cup {\rm supp}(a)|=j$, then $x\cdot a=j$. A simple combinatorial counting shows that
\begin{equation*}
  \lambda_a=\sum_{s\in S}(-1)^{s\cdot a}=\sum_{x\in V_n: wt(x)=n/2}(-1)^{x\cdot a}=\sum_{j=0}^{n/2}(-1)^j\tbinom{r}{j}\tbinom{n-r}{n/2-j}=K_{n/2}^n(r).
\end{equation*}
By using the Reciprocal Law of the Krawchouk polynomials (see Theorem \ref{Krawchouk}), we have
\begin{equation*}
  K_{n/2}^n(r)=\frac{\tbinom{n}{n/2}}{\tbinom{n}{r}}K_r^n(n/2).
\end{equation*}
Since $K_r^n(n/2)$ is the coefficient of $x^r$ in $(1-x)^{n/2}(1+x)^{n-n/2}=(1-x^2)^{n/2}$. Thus, if $r=2j+1$ is odd, then $\lambda_a=K_{n/2}^n(2j+1)=0$; if $r=2j$ for some $j$, then
\begin{equation}\label{f-19}
  \lambda_a=(-1)^j\frac{\tbinom{n}{n/2}\tbinom{n/2}{j}}{\tbinom{n}{2j}}.
\end{equation}
Now, it is easy to see that the minimum eigenvalue of $H_n$ is
\begin{equation}\label{f-19'}
 \lambda_{\min}=-\frac{\tbinom{n}{n/2}\tbinom{n/2}{1}}{\tbinom{n}{2}}=-\frac{\tbinom{n}{n/2}}{{n-1}}=-\frac{\lambda_{\max}}{{n-1}}.
\end{equation}
Then, by the spectral bounds in (\ref{f-4'}), we obtain
\begin{equation}\label{f-20}
 \chi_q(H_n)\geq 1+\frac{\lambda_{\max}}{|\lambda_{\min}|}=n.
\end{equation}
Next, we show that $\chi_q(H_n)\leq n$. To this end, we need to find a quantum homomorphism of $H_n$. Very recently, Menamara \cite{Mena} found such a homomorphism. We provide his result for completeness.

For every $x=(x_0,x_1,\cdots,x_{n-1})\in V_n$, and $0\leq \alpha\leq n-1$, we define the following operators:
\begin{equation}\label{f-21}
 P_x^\alpha=(a_x^{\alpha}(i,j))_{0\leq i,j\leq n-1}, a_x^\alpha(i,j)=\frac{1}{n}\xi_n^{(j-i)\alpha}(-1)^{x_i+x_j},
\end{equation}
where $\xi_n=e^{\frac{2 \pi \sqrt{-1}}{n}}$ is an $n$-th root of unity in $\mathbb{C}$. Then it is obvious that $P_x^\alpha$ is a Hermite matrix, moreover, let $(P_x^\alpha)^2=(b(i,j))_{0\leq i,j\leq n-1}$. Then
\begin{eqnarray*}
  b(i,j) &=& \sum_{k=0}^{n-1}a_x^\alpha(i,k)a_x^{\alpha}(k,j) \\
  &=& \frac{1}{n^2}\sum_{k=0}^{n-1}\xi_n^{(k-i)\alpha}(-1)^{x_i+x_k}\xi_n^{(j-k)\alpha}(-1)^{x_j+x_k}\\
  &=&\frac{1}{n^2}\xi_n^{(j-i)\alpha}(-1)^{x_i+x_j}\sum_{k=0}^{n-1}1\\
  &=&a_x^{\alpha}(i,j).
\end{eqnarray*}
Thus $(P_x^\alpha)^2=P_x^\alpha$. That is, $P_x^{\alpha}$ is a projection.

For every $x\in V_n$, let $\triangle_x=\{P_x^{\alpha}: 0\leq \alpha\leq n-1\}$. We aim to prove $\triangle_x$ is a complete orthogonal system of $\mathbb{C}^{n\times n}$. Indeed, for every $0\leq \alpha\neq \alpha'\leq n-1$, denote $P_x^\alpha P_x^{\alpha'}=(c(i,j))$. Then
\begin{eqnarray*}
  c(i,j) &=& \sum_{k=0}^{n-1}\frac{1}{n^2}\xi_n^{(k-i)\alpha}(-1)^{x_i+x_k}\xi_n^{(j-k)\alpha'}(-1)^{x_j+x_k} \\
  &=&\frac{1}{n^2}\xi_n^{j\alpha'-i\alpha}(-1)^{x_i+x_j}\sum_{k=0}^{n-1}\xi_n^{k(\alpha-\alpha')}\\
  &=&0.
\end{eqnarray*}
Therefore, $P_x^\alpha P_x^{\alpha'}=0$.
Furthermore, we can prove that for every $x\in V_n$, the above defined $\triangle_x$ is complete, i.e., $\sum_{\alpha=0}^{n-1}P_x^{\alpha}=I_n$. Let $\sum_{\alpha=0}^{n-1}P_x^{\alpha}=(u (i,j))_{0\leq i,j\leq n-1}$. Then
\begin{eqnarray*}
  u(i,j) &=& \sum_{\alpha=0}^{n-1}a_x^{\alpha}(i,j) \\
  &=&\frac{1}{n}\sum_{\alpha=0}^{n-1}\xi_n^{(j-i)\alpha}(-1)^{x_i+x_j}\\
  &=&\delta_{i,j}(-1)^{x_i+x_j}\\
  &=&\delta_{i,j},
\end{eqnarray*}
where $\delta_{i,j}=1$ if $i=j$, and $0$ otherwise.
Thus $\sum_{\alpha=0}^{n-1}P_x^{\alpha}=I_n$.

Finally, let $x,y\in V_n$ with $(x,y)\in E$ be an edge of $H_n$, that is $d(x,y)=2t$. Then
\begin{eqnarray*}
 ( P_x^{\alpha}P_y^\alpha)(i,j)&=&\frac{1}{n^2}\sum_{k=0}^{n-1}\xi_n^{(i-k)\alpha}(-1)^{x_i+x_k}\xi_n^{(k-j)\alpha}(-1)^{y_j+y_k}\\
 &=&\frac{1}{n^2}\xi_{n}^{(i-j)\alpha}(-1)^{x_i+y_j}\sum_{k=0}^{n-1}(-1)^{x_k+y_k}\\
 &=&\frac{1}{n^2}\xi_{n}^{(i-j)\alpha}(-1)^{x_i+y_j}(-2t+4t-2t)\\
 &=&0.
\end{eqnarray*}
Thus the set $\mathfrak{F}=\{\Delta_x: x \in V_n\}$ provides a quantum colouring of $H_n$.
Therefore, by the definition of quantum chromatic numbers, we know that
\begin{equation}\label{f-23}
   \chi_q(H_n)\leq n.
\end{equation}
Combining (\ref{f-20}) and (\ref{f-23}), we have $\chi_q(H_n)=n$ as required.

\subsection{Some new results}\label{neq results}

\subsubsection{Quantum chromatic numbers of a kind of Hamming graphs}

Firstly, we have the following result:

\begin{thm}\label{main-1}Let $V_n=\mathbb{F}_2^n$ be the $n$-dimensional vector space over $\mathbb{F}_2$, $S=\{x\in V_n: wt(s)=\ell\}$. Define a graph by $H_{n,\ell}:={\rm Cay}(V_n,S)$. If $n=4t-1$ and $\ell=2t$ for some positive integer $t$, then the spectrum of $H_{n,\ell}$ is
\begin{equation}\label{f-36}
   \lambda_a=\left\{\begin{array}{cl}
                    (-1)^j\frac{\tbinom{4t-1}{2t}\tbinom{2t-1}{j}}{\tbinom{4t-1}{2j}} & \mbox{ if $wt(a)=r=2j$, $0\leq j\leq 2t-1$,} \\
                    (-1)^{j+1}\frac{\tbinom{4t-1}{2t}{\tbinom{2t-1}{j}}}{\tbinom{4t-1}{2j+1}} & \mbox{ if $wt(a)=r=2j+1$, $0\leq j\leq 2t-1$}.
                  \end{array}
 \right.
\end{equation}
Moreover,
\begin{equation}\label{37}
  \chi_q(H_{n,\ell})=n+1.
\end{equation}
\end{thm}

\begin{proof} For every $a=(a_0,a_1,\cdots,a_{n-1})\in V_n$, if $wt(a)=r$, the corresponding eigenvalue of $H_{n,\ell}$ is
\begin{equation*}
  \lambda_a=\sum_{s\in S}(-1)^{s\cdot a}.
\end{equation*}
It is readily seen that the maximum eigenvalue of $H_{n,\ell}$ is $\lambda_{\max}=\tbinom{n}{\ell}=\lambda_{0_n}=\lambda_{1_n}$ since $\ell$ is even.

For $a(\neq 0_n,1_n)\in V_n$,
\begin{equation*}
  \lambda_a=\sum_{s\in S}(-1)^{s\cdot a}=\sum_{x\in V_n: wt(x)=\ell}(-1)^{x\cdot a}.
\end{equation*}
Moreover,
\begin{equation}\label{f-23'}
  \lambda_{1_n-a}=\sum_{s\in S}(-1)^{s\cdot (1_n-a)}=\sum_{x\in V_n: wt(x)=\ell}(-1)^{x\cdot (1_n-a)}=(-1)^\ell\lambda_a=\lambda_a.
\end{equation}
A similar analysis as that in Section \ref{proof of Thm-1} leads to
\begin{equation}\label{f-24}
  \lambda_a=\sum_{s\in S}(-1)^{s\cdot a}=\sum_{x\in V_n: wt(x)=\ell}(-1)^{x\cdot a}=K_\ell^n(r).
\end{equation}

By the Reciprocal Law of the Krawchouk polynomials (see Theorem \ref{Krawchouk}), we have
\begin{equation*}
  K_\ell^n(r)=\frac{\tbinom{n}{\ell}}{{\tbinom{n}{r}}}K_r^n(\ell).
\end{equation*}
Now, $K_r^n(\ell)$ is the coefficient of $x^r$ in the expansion of
\begin{equation*}
  (1-x)^\ell(1+x)^{n-\ell}=(1-x^2)^{2t-1}(1-x)=\sum_{j=0}^{2t-1}(-1)^j\tbinom{2t-1}{j}(x^{2j}-x^{2j+1}).
\end{equation*}
Therefore, if $r=2j$ for some $j$, then
\begin{equation}\label{f-25}
   \lambda_a=(-1)^j\frac{\tbinom{4t-1}{2t}\tbinom{2t-1}{j}}{\tbinom{4t-1}{2j}}.
\end{equation}
If  $r=2j+1$ is odd, then
\begin{equation}\label{f-26}
   \lambda_a=(-1)^{j+1}\frac{\tbinom{4t-1}{2t}\tbinom{2t-1}{j}}{\tbinom{4t-1}{2j+1}}.
\end{equation}
By (\ref{f-23'}), (\ref{f-25}) and (\ref{f-26}), one can check that
\begin{equation}\label{f-27}
   \lambda_{\min}=-\frac{\tbinom{4t-1}{2t}}{4t-1}=-\frac{\lambda_{\max}}{4t-1}.
\end{equation}

In fact, by (\ref{f-25}) and (\ref{f-26}), $\lambda_a$ depends only on $wt(a)=r$. Write $\lambda_a=\rho(r)$. To find $\lambda_{\min}$, we just need to consider $\rho(4j+2)$ and $\rho(4j+1)$ for $0\leq j\leq \lfloor \frac{t}{2}\rfloor$.

Now,
\begin{equation*}
  \frac{|\rho(4j+2)|}{|\rho(4j+1)|}=\frac{\tbinom{2t-1}{2j+1}\tbinom{4t-1}{4j+1}}{\tbinom{2t-1}{2j}\tbinom{4t-1}{4j+2}}=\frac{(4j+2)!(4t-4j-3)!(2j)!(2t-2j-1)!}{(4j+1)!(4t-4j-2)!(2j+1)!(2t-2j-2)!}=1.
\end{equation*}
Moreover, for $1\leq j\leq \lfloor\frac{t}{2}\rfloor$,
\begin{equation*}
  \frac{|\rho(4j+2)|}{|\rho(4j-2)|}=\frac{\tbinom{2t-1}{2j+1}\tbinom{4t-1}{4j-2}}{\tbinom{2t-1}{2j-1}\tbinom{4t-1}{4j+2}}<1.
\end{equation*}
So that the eigenvalues of $H_{4t-1,2t}$ with negative sign are
\begin{equation*}
  \rho(1)=\rho(2)<\rho(5)=\rho(6)<\cdots<\rho\left(4\left\lfloor \frac{t}{2}\right\rfloor-3\right)=\rho\left(4\left\lfloor \frac{t}{2}\right\rfloor-2\right).
\end{equation*}
Note that we just list one half of them, the symmetric part is $\rho(n-1)=\rho(1)$, $\rho(n-5)=\rho(5)$ and so on.
Thus, $\lambda_{\min}=\rho(1)=-\frac{\tbinom{4t-1}{2t}}{4t-1}=-\frac{\lambda_{\max}}{4t-1}.$

Thanks to Lemma \ref{lem-2.4}, we have
\begin{equation}\label{f-28}
  \chi_q(H_{4t-1,2t})\geq 4t=n+1.
\end{equation}

In order to prove $\chi_q(H_{4t-1,2t})\leq 4t$, we need to find a proper quantum colouring of such a graph. To this end, for every $x=(x_0,x_1,\cdots,x_{n-1})\in V_n$, we embed it to $V_{n+1}$ by adding a coordinate at the end of $x$ as $\widetilde{x}=(x_0,x_1,\cdots,x_{n-1},0)$, i.e., $x_n=0$. Define the following set of operators on $\mathbb{C}^{4t}$:
\begin{equation*}
  \mathfrak{F}:=\{P_x^{\alpha}: x\in V_n, 0\leq \alpha\leq 4t-1\},
\end{equation*}{ where }
\begin{equation*}
  P_x^{\alpha}=(p_x^{\alpha}(i,j))_{0\leq i,j\leq 4t-1}, p_x^{\alpha}(i,j)=\frac{1}{4t}\xi_{4t}^{(j-i)\alpha}(-1)^{x_i+x_j}.
\end{equation*}
It is obvious that each operator is Hermitian, and the $(i,j)$-th entry of $(P_x^\alpha)^2$ is
\begin{equation*}
  ((P_x^\alpha)^2){(i,j)}=\sum_{k=0}^{4t-1}p_x^\alpha(i,k)p_x^\alpha(k,j)=\frac{1}{(4t)^2}\xi_{4t}^{(j-i)\alpha}(-1)^{x_i+x_j}\sum_{k=0}^{4t-1}1= p_x^{\alpha}(i,j).
\end{equation*}
Thus $P_x^\alpha$ is a projection. Similarly, one can prove that
\begin{equation*}
  P_x^\alpha P_x^{\alpha'}=0 \mbox{ if } \alpha\neq \alpha',
\end{equation*}
\begin{equation*}
\mbox{ and} \sum_{0\leq \alpha\leq 4t-1}P_x^{\alpha}=I_{4t}.
\end{equation*}
 Now, we assume that $(x,y)\in E_{4t-1,2t}$ is an edge of $H_{4t-1,2t}$, i.e., $d(x,y)=2t$, then we need to prove that $P_x^\alpha P_y^\alpha=0$, whence $x=(x_0,x_1,\cdots,x_{4t-2}), y=(y_0,y_1,\cdots,y_{4t-2})$ and $x_{4t-1}=y_{4t-1}=0$.
 Since
 \begin{equation*}
   \sum_{k=0}^{4t-1}(-1)^{x_k+y_k}=1+\sum_{k=0}^{4t-2}(-1)^{x_k+y_k}=1-2t+2t-1=0,
 \end{equation*}
 it follows that
 \begin{equation*}
   (P_x^\alpha P_y^\alpha)_{i,j}=\sum_{k=0}^{4t-1}p_x^{\alpha}(i,k)p_y^{\alpha}(k,j)=\frac{1}{(4t)^2}\xi_{4t}^{(j-i)\alpha}(-1)^{x_i+y_j}\sum_{k=0}^{4t-1}(-1)^{x_k+y_k}=0.
 \end{equation*}
 Thus $P_x^{\alpha}P_y^{\alpha}=0$ if $(x,y)\in  E_{4t-1,2t}$ is an edge.

 Therefore, $\mathfrak{F}$ gives a quantum colouring of $H_{4t-1,2t}$, and then
 \begin{equation}\label{f-30}
    \chi_q(H_{4t-1,2t})\leq 4t=n+1.
 \end{equation}
 Combining (\ref{f-28}) and (\ref{f-30}) yields $\chi_q(H_{4t-1,2t})=4t=n+1$.
 This completes the proof.
\end{proof}

\begin{remark}We note that, except for some trivial cases, the family of graphs in Theorem \ref{main-1} is the second known family of infinite graphs whose quantum chromatic numbers are exactly determined, up to our knowledge.\end{remark}

\subsubsection{A upper bound for quantum chromatic numbers of Hamming graphs}

Let $\ell$ be a positive integer. $H_{n,\ell}={\rm Cay}(V_n,S)$, where $S=\{x\in V_n: wt(x)=\ell\}$. If $\ell$ is odd, then $H_{n,\ell}$ is a bipartite graph and then $\chi_q(H_{n,\ell})=2$. Thus we assume that $\ell=2t$ for some integer $t$. In this case, we have the following result:

\begin{thm}\label{thm-3.5}Let notations be defined as above, if $\ell\geq n/2$, then
\begin{equation}\label{f-31}
  \chi_q(H_{n,\ell})\leq 2\ell.
\end{equation}
\end{thm}
\begin{proof}Denote $d=2\ell(\geq n)$. For every $x=(x_0,x_1,\cdots,x_{n-1})\in V_n$, we let
$x_{n}=\cdots =x_{d-1}=0$ if $d>n$.
Define a set of operators as follows:
\begin{equation*}
  \mathfrak{F}=\{P_x^{\alpha}: x\in V_n, 0\leq \alpha\leq d-1\},
\end{equation*}
where $P_x^{\alpha}(i,j)=\frac{1}{d}\xi_d^{(i-j)\alpha}(-1)^{x_i+x_j}, 0\leq i,j\leq d-1$.

Using a similar method as that in the previous section, we can prove that $\mathfrak{F}$ forms a complete orthogonal system of $\mathbb{C}^d$.

When $0\leq \alpha\leq d-1$, we claim that if $x,y\in V_n$ with $d(x,y)=\ell$, it holds that $P_x^{\alpha}P_y^{\alpha}=0$. Note that for $y\in V_n$, we define $y_{n}=\cdots=y_{d-1}=0$. Indeed, we have
\begin{equation*}
  P_x^{\alpha}P_y^{\alpha}(i,j)=\frac{1}{d^2}\sum_{k=0}^{d-1}\xi_d^{(i-k)\alpha}(-1)^{x_i+x_k}\xi_d^{(j-k)\alpha}(-1)^{y_k+y_j}=\frac{1}{d^2}\xi_d^{(j-i)\alpha}(-1)^{x_i+x_j}\sum_{k=0}^{d-1}(-1)^{x_k+y_k}.
\end{equation*}
Now,
\begin{equation*}
  \sum_{k=0}^{d-1}(-1)^{x_k+y_k}= \sum_{k=0}^{n-1}(-1)^{x_k+y_k}+ \sum_{k=n}^{d-1}(-1)^{x_k+y_k}=(-\ell+n-\ell)+(d-n)=0.
\end{equation*}
Therefore, $\mathfrak{F}$ gives a proper quantum colouring of $H_{n,\ell}$, and then the desired result follows.
\end{proof}
Particularly, if $n=4t+2$, $\ell=2t+2$, we have
\begin{prop}\label{prop-3.6} Let $H_{n,\ell}={\rm Cay}(V_n,S)$, where $S=\{x\in V_n: wt(x)=\ell\}$. If $n=4t+2,\ell=2t+2$, then the spectrum of $H_{n,\ell}$ is given by
\begin{equation}\label{f-33}
 \lambda_a=\left\{\begin{array}{cl}
                    (-1)^j\frac{\left[\tbinom{2t}{j}-\tbinom{2t}{j-1}\right]{\tbinom{n}{\ell}}}{\tbinom{n}{2j}} & \mbox{ if $wt(a)=r=2j$,} \\
                    (-1)^{j+1}\frac{2\tbinom{2t}{j}\tbinom{n}{\ell}}{\tbinom{n}{2j+1}} & \mbox{ if $wt(a)=r=2j+1$}.
                  \end{array}
 \right.
\end{equation} Moreover,
\begin{equation}\label{f-34}
 \ell\leq \chi_q(H_{n,\ell})\leq 2\ell.
\end{equation}
\end{prop}
\begin{proof}Since $\ell>n/2$, the upper bound $\chi_q(H_{n,\ell})\leq 2\ell$ follows from Theorem \ref{thm-3.5}. Now, we prove the lower bound. In order to prove it, we compute the eigenvalues of $H_{n,\ell}$.
For every $a\in V_n$, if $wt(a)=r$, the eigenvalue of $H_{n,\ell}$ corresponding to $a$ is
\begin{equation*}
  \lambda_a=\sum_{x\in V_n, wt(x)=\ell}(-1)^{x \cdot a}=K_\ell^n(r)=\frac{\tbinom{n}{\ell}}{\tbinom{n}{ r}}K_r^n(\ell).
\end{equation*}
When $n=4t+2,\ell=2t+2$, $K_r^n(\ell)$ is the coefficient of $x^r$ in the expansion of
\begin{equation*}
   (1-x)^{2t+2}(1+x)^{2t} =1+x^{4t+2}+\sum_{j=1}^{2t}(-1)^j\left[\tbinom{2t}{j}-\tbinom{2t}{j-1}\right]x^{2j}-2\sum_{j=1}^{2t}(-1)^j\tbinom{2t}{j}x^{2t+1}.
\end{equation*}
Thus, we have
\begin{equation}\label{f-35}
 \lambda_a=\left\{\begin{array}{cl}
                    (-1)^j\frac{\left[\tbinom{2t}{j}-\tbinom{2t} {j-1}\right]\tbinom{n}{\ell}}{\tbinom{n}{2j}} & \mbox{ if $wt(a)=r=2j$,} \\
                    (-1)^{j+1}\frac{2\tbinom{2t}{j}\tbinom{n}{\ell}}{\tbinom{n}{2j+1}} & \mbox{ if $wt(a)=r=2j+1$}.
                  \end{array}
 \right.
\end{equation}
It is a routine to check that $\lambda_{\min}=-\tbinom{n }{\ell}/(2t+1)$ as we have done in previous sections. By the spectral bounds on the quantum chromatic numbers, see Lemma \ref{lem-2.4}, we have
\begin{equation*}
  \chi_q(H_{4t+2,2t+2})\geq 2t+2=\ell.
\end{equation*}
This completes the proof.
\end{proof}

\begin{remark}Form the discussions before this, we can see that for the Hamming graphs $H_{n,\ell}$, we can use the properties of Krawchouk polnomials to evaluate the spectra of these graphs, and then by the spectral bounds on quantum chromatic numbers, we can get a lower bound on such numbers, especially, when $n-2\ell$ is small. If $\ell>n/2$, we have a upper bound for the quantum chromatic numbers. But for $\ell<n/2$, we don't know how to find such a upper bound. The following question is still open.

\noindent{\bf Open question}: Find a upper bound on $\chi_q(H_{n,\ell})$ when $2\ell<n$.
\end{remark}

\subsection{Quantum chromatic numbers of products of graphs}

Let $\Gamma_1$ and $\Gamma_2$ be two graphs. The product of $\Gamma_1$ and $\Gamma_2$, denoted by $\Gamma_1\times \Gamma_2$, is defined by $\Gamma_1\times \Gamma_2=(V,E)$, where
$V=V(\Gamma_1)\times V(\Gamma_2)=\{(x_1,x_2): x_1\in V(\Gamma_1), x_2\in V(\Gamma_2)$ and $E=\{((x_1,x_2),(y_1,y_2)): (x_1,y_1)\in E(\Gamma_1) \mbox{ and } (x_2,y_2)\in E(\Gamma_2)\}$. Let the eigenvalues of $\Gamma_1$ be $\lambda_1\geq \cdots \geq \lambda_n$ and the eigenvalues of $\Gamma_2$ be $\mu_1\geq \cdots \geq \mu_m$. Then the eigenvalues of
$\Gamma_1\times \Gamma_2$ are $\lambda_i\mu_j: 1\leq i\leq n, 1\leq j\leq m$. Thus $\lambda_{\max}(\Gamma_1\times \Gamma_2)=\lambda_1\mu_1$, and $\lambda_{\min}(\Gamma_1\times \Gamma_2)=\min\{\lambda_1\mu_m,\lambda_n\mu_1\}$.
In \cite{SM}, a upper bound on the quantum chromatic numbers of products of graphs is given as follows:

\begin{lem}\cite{SM} \label{lem-3.8} Let notations be defined as above. Then
\begin{equation}\label{f-38}
 \chi_q(\Gamma_1\times \Gamma_2)\leq \min\{\chi_q(\Gamma_1), \chi_q(\Gamma_2)\}.
\end{equation}
\end{lem}
By Lemma \ref{lem-3.8}, Elphick and Wocjan \cite{CW} proved the following result.

\begin{thm}[\cite{CW}, Prop. 3.2]\label{thm-3.9} Suppose that the eigenvalues of $\Gamma_1$ are $\lambda_1\geq \cdots \geq \lambda_n$ and the eigenvalues of $\Gamma_2$ are $\mu_1\geq \cdots \geq \mu_m$. If $\frac{\lambda_1}{|\lambda_n|}\geq \frac{\mu_1}{|\mu_m|}$, then
\begin{equation}\label{f-40}
  \chi_q(\Gamma_1\times \Gamma_2)\geq 1+\frac{\mu_1}{|\mu_m|}.
\end{equation}
\end{thm}
As a consequence, we have
\begin{cor}\label{cor-3.10} If $\chi_q(\Gamma_1)=1+\frac{\lambda_1}{|\lambda_n|}$ and $\chi_q(\Gamma_2)=1+\frac{\mu_1}{|\mu_m|}$, then
\begin{equation}\label{f-41}
   \chi_q(\Gamma_1\times \Gamma_2)=\min\{\chi_q(\Gamma_1), \chi_q(\Gamma_2)\}.
\end{equation}
\end{cor}

By Corollary \ref{cor-3.10}, we get the following result for products of Hamming graphs.
\begin{thm}\label{thm-3.11} Let $n_i,\ell_i$ $(i=1,2)$ and $s\geq t$ be positive integers. Then
\begin{enumerate}
  \item $\chi_q(H_{4t,2t}\times H_{4s,2s})=4t$;
  \item $\chi_q(H_{4t-1,2t}\times H_{4s,2s})=4t$;
  \item $\chi_q(H_{4t,2t}\times H_{4s-1,2s})=4t$;
  \item $\chi_q(H_{4t-1,2t}\times H_{4s-1,2s})=4t$.
\end{enumerate}
\end{thm}

\section{Appendix}
For readers convenience, we provide some tables on the spectra of some Hamming graphs. Note that if $wt(a)=r$, we denote by $\rho_\ell^n(r)(=K_\ell^n(r))$ the eigenvalue of $H_{n,\ell}$ corresponding to $a$.

Table 1. $\rho_{\ell}^3(r)$

\begin{tabular}{|c|c|c|c|c|}
  \hline
  $n=3$ & $\ell=0$ & 1 & 2 & 3 \\
   \hline
 $r=0$ & 1 & 3 & 3 & 1 \\
  1 & 1 & 1 & -1 & -1 \\
  2 & 1 & -1 & -1 & 1 \\
  3 & 1 & -3 & 3 & -1 \\
  \hline
\end{tabular}

\vskip 0.2 cm
Table 2. $\rho_\ell^4(r)$

\begin{tabular}{|c|c|c|c|c|c|}
  \hline
  $n=4$ & $\ell=0$ & 1 & 2 & 3 & 4 \\
  \hline
  $r=0$ & 1 & 4 & 6 & 4 & 1 \\
  1 & 1 & 2 & 0 & -2 & -1 \\
  2 & 1 & 0 & -2 & 0 & 1 \\
  3 & 1 & -2 & 0 & 2 & -1 \\
  4 & 1 & -4 & 6 & -4 & 1 \\
  \hline
\end{tabular}

\vskip 0.2 cm
Table 3. $\rho_\ell^5(r)$

\begin{tabular}{|c|c|c|c|c|c|c|}
  \hline
 $n=5$ & $\ell=0$ & 1 & 2 & 3 & 4 & 5 \\
  \hline
 $r=0$ & 1 & 5 & 10 & 10 & 5 & 1 \\
  1 & 1 & 3 & 2 & -2 & -3 & -1 \\
  2 & 1 & 1 & -2 & -2 & 1 & 1 \\
  3 & 1 & -1 & -2 & 2 & 1 & -1 \\
  4 & 1 & -3 & 2 & 2 & -3 & 1 \\
  5 & 1 & -5 & 10 & -10 & 5 & -1 \\
  \hline
\end{tabular}

\vskip 0.2 cm
Table 4. $\rho_\ell^6(r)$

\begin{tabular}{|c|c|c|c|c|c|c|c|}
  \hline
 $n=6$ & $\ell=0$ & 1 & 2 & 3 & 4 & 5 &6\\
  \hline
 $r=0$ & 1 & 6 & 15 & 20 & 15 & 6 &1\\
  1 & 1 & 4 & 5 & 0 & -5 & -4& -1\\
  2 & 1 & 2 & -1 & -4 & -1 & 2 &1\\
  3 & 1 & 0 & -3 & 0 & 3 & 0 &-1\\
  4 & 1 & -2 & -1 & 4 & -1 & -2&1 \\
  5 & 1 & -4 & 5 & 0 & -5 & 4& -1\\
  6 & 1 & -6 & 15 & -20 & 15 & -6&1 \\
  \hline
\end{tabular}

\newpage
Table 5. $\rho_\ell^7(r)$

\begin{tabular}{|c|c|c|c|c|c|c|c|c|}
  \hline
 $n=7$ & $\ell=0$ & 1 & 2 & 3 & 4 & 5 &6&7\\
  \hline
 $r=0$ & 1 & 7 & 21 & 35 & 35 & 21 &7&1\\
  1 & 1 & 5 & 9 & 5 & -5 & -9& -5&-1\\
  2 & 1 & 3 & 1 & -5 & -5 & 1 &3&1\\
  3 & 1 & 1 & -3 & -3 & 3 & 3 &-1&-1\\
  4 & 1 & -1 & -3 & 3 & 3 & -3&-1& 1\\
  5 & 1 & -3 & 1 & 5 & -5 & -1& 3&-1\\
  6 & 1 & -5 & 9 & -5 & -5 & 9&-5 &1\\
  7 & 1 & -7 & 21 & -35 & 35 & -21&7 &-1\\
  \hline
\end{tabular}

\vskip 0.2 cm
Table 6. $\rho_\ell^8(r)$

\begin{tabular}{|c|c|c|c|c|c|c|c|c|c|}
  \hline
 $n=8$ & $\ell=0$ & 1 & 2 & 3 & 4 & 5 &6&7&8\\
  \hline
 $r=0$ & 1 & 8 & 28 & 56 &70 & 56 &28&8&1\\
  1 & 1 & 6 & 14 & 14 & 0 & -14& -14&-6&-1\\
  2 & 1 & 4 & 4 & -4 & -10 & -4 &4&4&1\\
  3 & 1 & 2 & -2 & -6 & 0 & 6 &2&-2&-1\\
  4 & 1 & 0 & -4 & 0 & 6 & 0&-4& 0&1\\
  5 & 1 & -2 & -2 & 6 & 0& -6& 2&2&-1\\
  6 & 1 & -4 & 4 & 4 & -10 & 4&4 &-4&1\\
  7 & 1 & -6 & 14 & -14 & 0 & 14&-14 &6&-1\\
  8 & 1 & -8 & 28 & -56 & -70 & -56&28 &-8&1\\
  \hline
\end{tabular}

\vskip 0.2 cm
Table 7. $\rho_\ell^9(r)$

\begin{tabular}{|c|c|c|c|c|c|c|c|c|c|c|}
  \hline
 $n=9$ & $\ell=0$ & 1 & 2 & 3 & 4 & 5 &6&7&8&9\\
  \hline
 $r=0$ & 1 & 9 & 36 & 84 &126 & 126 &84&36&9&1\\
  1 & 1 & 7 & 20 & 28 & 14 & -14& -28&-20&-7&-1\\
  2 & 1 & 5 & 8 & 0 & -14 & -14 &0&8&5&1\\
  3 & 1 & 3 & 0 & -8 & -6 & 6 &8&0&-3&-1\\
  4 & 1 & 1 & -4 & -4 & 6 & 6&-4& -4&1&1\\
  5 & 1 & -1 & -4 &4 &6 & -6& -4& 4&1&-1\\
  6 & 1 & -3 & 0 & 8 & -6 & -6&8 &0&-3&1\\
  7 & 1 & -5 & 8 & 0 & -14 & 14&0 &-8&5&-1\\
  8 & 1 & -7 & 20 & -28 & 14 & 14&-28 &20&-7&1\\
  9 & 1 & -9 & 36 & -84 & 126 & -126&84 &-36&9&-1\\
  \hline
\end{tabular}

\newpage
Table 8. $\rho_\ell^{10}(r)$

\begin{tabular}{|c|c|c|c|c|c|c|c|c|c|c|c|}
  \hline
 $n=10$ & $\ell=0$ & 1 & 2 & 3 & 4 & 5 &6&7&8&9&10\\
  \hline
 $r=0$ & 1 & 10 & 45 & 120 &210 & 252 &210&120&45&10&1\\
  1 & 1 & 8 & 27 & 48 & 42 & 0& -42&-48&-27&-8&-1\\
  2 & 1 & 6 & 13 & 8 & -14 & -28 &-14&8&13&6&1\\
  3 & 1 & 4 & 3 &-8& -14 & 0 & 14 &8&-3&-4&-1\\
  4 & 1 & 2 & -3 & -8 & 2 & 12&2& -8&-3&2&1\\
  5 & 1 & 0 & -5 &0 &10& 0& -10& 0&5&0&-1\\
  6 & 1 & -2 & -3 & 8 & 2 & -12&2 &8&-3&-2&1\\
  7 & 1 & -4 & 3 & 8 &-14& 0 & 14&-8 &-3&4&-1\\
  8 & 1 & -6 & 13 & -8 & -14 & 28&-14 &-8&13&-6&1\\
  9 & 1 & -8 & 27 & -48 & 42 & 0&-42 &48&-27&8&-1\\
  10 & 1 & -10 & 45 & -120 & 210 & -252&210 &-120&45&-10&1\\
  \hline
\end{tabular}

\vskip 0.2 cm
Table 9. $\rho_\ell^{11}(r)$

\begin{tabular}{|c|c|c|c|c|c|c|c|c|c|c|c|c|}
  \hline
 $n=11$ & $\ell=0$ & 1 & 2 & 3 & 4 & 5 &6&7&8&9&10&11\\
  \hline
 $r=0$ & 1 & 11 & 55 & 165 &330 & 462 &462&330&165&55&11&1\\
  1 & 1 & 9 & 35 & 75 & 90 & 42& -42&-90&-75&-35&-9&-1\\
  2 & 1 & 7 & 19 & 21 & -6 & -42 &-42&-6&21&19&7&1\\
  3 & 1 & 5 & 7 &-5& -22 & -14 & 14 &22&5&-7&-5&-1\\
  4 & 1 & 3 & -1 & -11 & -6 & 14&14& -6&-11&-1&3&1\\
  5 & 1 & 1 & -5 &-5 &10& 10& -10& -10&5&5&-1&-1\\
  6 & 1 & -1 & -5 & 5 & 10 & -10&-10 &10&5&-5&-1&1\\
  7 & 1 & -3 & -1 & 11 &-6& -14 & 14&6 &-11&1&3&-1\\
  8 & 1 & -5 & 7 & 5 & -22 & 14&14 &-22&5&7&-5&1\\
  9 & 1 & -7 & 19 & -21 & -6 & 42&-42 &6&21&-19&7&-1\\
  10 & 1 & -9 & 35 & -75 & 90 & -42&-42 &90&-75&35&-9&1\\
  11 & 1 & -11 & 55 & -165 & 330 & -462&462 &-330&165&-55&11&-1\\
  \hline
\end{tabular}

\vskip 0.2 cm
Table 10. $\rho_\ell^{12}(r)$

\begin{tabular}{|c|c|c|c|c|c|c|c|c|c|c|c|c|c|}
  \hline
 $n=12$ & $\ell=0$ & 1 & 2 & 3 & 4 & 5 &6&7&8&9&10&11 &12\\
  \hline
 $r=0$ & 1 & 12 & 66 & 220 &495 & 792 &924&792&495&220&66&12 &1\\
  1 & 1 & 10 & 44 & 110 & 165 & 132& 0&-132&-165&-110&-44&-10 &-1\\
  2 & 1 & 8 & 26 & 40 & 15 & -48 &-84&-48&15&40&26&8 &1\\
  3 & 1 & 6 & 12 &2& -27 & -36 & 0 &36&27&-2&-12&-6 &-1\\
  4 & 1 & 4 & 2 & -12 & -17 & 8&28& 8&-17&-12&2&4 &1\\
  5 & 1 & 2 & -4 &-10 &5& 20& 0& -20&-5&10&4& -2&-1\\
  6 & 1 &0 & -6 & 0 & 15 & 0&-20 &0&15&0&-6&0 &1\\
  7 & 1 & -2 & -4 & 10 &5& -20 & 0&20&-5 &-10&4&2& -1\\
  8 & 1 & -4 & 2 & 12 & -17 & -8&28 &-8&-17&12&2&-4 &1\\
  9 & 1 & -6 & 12 & -2 & -27 & 36&0 &-36&27&2&-12& 6&-1\\
  10 & 1 & -8 & 26 & -40 & 15 & 48&-84 &48&15&-40&26&-8 &1\\
  11 & 1 & -10 & 44 & -110 & 165 & -132&0 &132&-165&110&-44&10 &-1\\
  12 & 1 & -12 & 66 & -220 & 495 & -792&924 &-792&495&-220&66& -12&1\\
  \hline
\end{tabular}

\end{document}